\pdfoutput=1
\documentclass[a4paper,12pt,oneside,notitlepage]{article}
\usepackage[utf8]{inputenc}
\usepackage[T1]{fontenc}
\usepackage{amsmath,amsthm,amssymb,amsfonts}
\usepackage{mathtools,esint}
\usepackage{microtype}
\usepackage[charter]{mathdesign}
\usepackage[unicode]{hyperref}
\hypersetup{
bookmarks=true,
colorlinks=true,
citecolor=[rgb]{0,0,0.75},
linkcolor=[rgb]{0,0,0.75},
urlcolor=blue,
pdfpagemode=UseNone,
pdfstartview=FitH,
pdfdisplaydoctitle=true,
pdftitle={Intrinsic square functions with arbitrary aperture},
pdfauthor={Pavel Zorin-Kranich},
pdflang=en-US
}

\usepackage[textwidth=16cm,textheight=24cm]{geometry}

\usepackage[safeinputenc=true,style=alphabetic,firstinits,maxnames=5,doi=false,eprint=true,isbn=false,url=false]{biblatex}

\newbibmacro{string+doiurl}[1]{%
  \iffieldundef{doi}{%
    \iffieldundef{url}{%
         #1%
      }{%
      \href{\thefield{url}}{#1}%
    }%
  }{%
    \href{http://dx.doi.org/\thefield{doi}}{#1}%
  }%
}
\DeclareFieldFormat[article,misc]{title}{\usebibmacro{string+doiurl}{\mkbibquote{#1}}}

\numberwithin{equation}{section}
\newtheorem{theorem}[equation]{Theorem}
\newtheorem{lemma}[equation]{Lemma}
\newtheorem{corollary}[equation]{Corollary}
\theoremstyle{definition}

\theoremstyle{remark}
\newtheorem{remark}[equation]{Remark}

\DeclarePairedDelimiter\abs{\lvert}{\rvert}
\newcommand*{\meas}[1]{\mu(#1)}
\DeclarePairedDelimiter\norm{\lVert}{\rVert}
\DeclarePairedDelimiter\Set\{\}
\newcommand*{\N}{\mathbb{N}}
\newcommand*{\Z}{\mathbb{Z}}
\newcommand*{\R}{\mathbb{R}}
\renewcommand*{\C}{\mathbb{C}}
\newcommand*{\dif}{\mathrm{d}}
\newcommand*{\dist}{\mathrm{dist}}
\newcommand*{\diam}{\mathrm{diam}}
\newcommand*{\supp}{\mathrm{supp}}
\newcommand*{\ch}{\mathrm{ch}}

\newcommand{\Dini}{\mathrm{Dini}}
\newcommand{\logDini}{\mathrm{log-Dini}}
\newcommand{\calD}{\mathcal{D}}
\newcommand{\calP}{\mathcal{P}}
\newcommand{\calQ}{\mathcal{Q}}
\newcommand{\sparse}{\mathcal{S}}
\newcommand*{\mean}[2][Q]{(#2)_{#1}}
\newcommand{\scale}{k}
\newcommand{\one}{\mathbf{1}}

\begin{document}
\allowdisplaybreaks
\title{Intrinsic square functions with arbitrary aperture}
\author{Pavel Zorin-Kranich}
\date{}
\maketitle
\begin{abstract}
We consider intrinsic square functions defined using (log-)Dini continuous test functions on spaces of homogeneous type.
We prove weighted estimates with optimal (at least in the Euclidean case) dependence on the aperture of the cone used to define the square function and linear dependence on the (log-)Dini modulus of continuity.
\end{abstract}

\section{Introduction}
Let $(X,\rho,\mu)$ be a space of homogeneous type, $\kappa>1$ a scale parameter, and $\omega$ a \emph{modulus of continuity}, that is, a function $\omega :[0,\infty) \to [0,\infty)$ with $\omega(0)=0$ that is subadditive in the sense that $u\leq t+s \implies \omega(u)\leq \omega(t)+\omega(s)$.
The \emph{Dini} and \emph{log-Dini} norms of a modulus of continuity are defined by
\begin{equation}
\label{eq:Dini}
\norm{\omega}_{\Dini}
:=
\int_{0}^{1} \omega(t) \frac{\dif t}{t},
\quad
\norm{\omega}_{\logDini}
:=
\int_{0}^{1} \omega(t) \frac{\abs{\log t} \dif t}{t}.
\end{equation}
Fix a parameter $\Phi>0$.
For a locally integrable function $f\in L^{1}_{\mathrm{loc}}(X)$ define a function on $X\times\Z$ by
\[
A_{\omega}f(x,k) := \sup_{\phi\in C_{\omega}(x,k)} \abs[\big]{\int f \phi},
\]
where the class of test functions is given by
\begin{multline*}
C_{\omega}(x,k) := \Set[\Big]{ \phi : X\to\R, \quad \supp\phi\subset B(x,\kappa^{k}),\ \int\phi=0,\ \sup_{y} \abs{\phi(y)} \leq \meas{B(x,\kappa^{k})}^{-1} \Phi,\\
\forall y,y'\ \abs{\phi(y)-\phi(y')}\leq \meas{B(x,\kappa^{k})}^{-1} \omega(\rho(y,y')/\kappa^{k}) }.
\end{multline*}
The \emph{intrinsic square function with aperture $\beta$} is defined by
\begin{equation}
\label{eq:G}
G_{\omega,\beta}f(x) := \Big( \sum_{k\in\Z} \fint_{B(x,\beta\kappa^{k})} (A_{\omega}f(y,k))^{2} \dif y \Big)^{1/2},
\end{equation}
where $\fint_{Q}f = (f)_{Q} = \meas{Q}^{-1} \int_{Q} f$ denotes an average.
\begin{theorem}
\label{thm:L1-FS}
Let $(X,\rho,\mu)$ be a space of homogeneous type, $\omega$ be a modulus of continuity, and $\beta\geq 1$.
Then for every non-negative locally integrable function (weight) $v\in L^{1}_{\mathrm{loc}}(X)$ we have
\begin{equation}
\label{eq:L2-FS}
\int_{X} (G_{\omega,\beta}f)^{2} v
\lesssim
\Phi \norm{\omega}_{\Dini} \int_{X} \abs{f}^{2} Mv
\end{equation}
and
\begin{equation}
\label{eq:isq-FS}
\sup_{\lambda>0} \lambda v \Set{ G_{\omega,\beta}f > \lambda }
\lesssim
\gamma^{1/2} \norm{\omega}_{\logDini} \int_{X} \abs{f} Mv,
\end{equation}
where $M$ denotes the uncentered Hardy--Littlewood maximal function and
\begin{equation}
\label{eq:doubling}
\gamma = \sup_{x\in X, t>0} \frac{\meas{B(x,\beta t)}}{\meas{B(x,t)}}.
\end{equation}
The implicit constants in \eqref{eq:L2-FS} and \eqref{eq:isq-FS} and later may depend on the quasi-triangle and the doubling constant of $X$ and the parameter $\kappa$, but not on $v,\omega,\beta,\Phi,f$.

If the reverse doubling property $B(x,CR) \geq (1+\epsilon) B(x,R)$ holds for some $C>1$, $\epsilon>0$, and all $x\in X$ and $R>0$, then the log-Dini norm in \eqref{eq:isq-FS} can be replaced by the Dini norm.
\end{theorem}
Here and later $A\lesssim B$ means $A\leq CB$ and $C$ denotes an absolute constant that may change from line to line.

In order to compare Theorem~\ref{thm:L1-FS} with earlier results notice that $\gamma\sim\beta^{D}$ if $(X,\rho,\mu)$ is Ahlfors--David $D$-regular, e.g.\ if $X=\R^{D}$.
The estimate \eqref{eq:L2-FS} for classical square functions goes back to \cite{MR800004} and the estimate \eqref{eq:isq-FS} to \cite{MR891775}.
The intrinsic function with a power modulus of continuity $\omega(t)=t^{\alpha}$, $0<\alpha\leq 1$, has been introduced by Wilson who proved the above estimates in the case $\beta=1$ \cite{MR2414491}.
The main novelty of Theorem~\ref{thm:L1-FS} is that we consider intrinsic square functions with varying aperture $\beta$, which in the weighted case is more efficient than reducing to the case $\beta=1$.
In the unweighted case $v\equiv 1$ Theorem~\ref{thm:L1-FS} can be deduced from the special case $\beta=1$ by \cite[Lemma 2.1]{MR3232586} (for the weak type $(1,1)$ estimate) and \cite{MR2783323} (for $1<p\leq 2$), at least on $X=\R^{D}$.
The dependence of weighted estimates on the aperture $\beta$ has been previously investigated for several more classical square functions in \cite{MR3232586,MR3521084}.

Without loss of generality we may assume $\Phi \lesssim \omega(1) \lesssim \norm{\omega}_{\Dini}$, so that $\Phi$ can be replaced by either of these quantities in \eqref{eq:L2-FS}.
An estimate involving $\Phi$ is useful in applications to singular integral operators.

The quasimetric case of Theorem~\ref{thm:L1-FS} follows from the metric case and the metrization theorem for quasimetric spaces \cite{MR2538591}, hence we assume from now on that $(X,\rho)$ is a metric space and $\mu$ a doubling measure: $\mu(B(x,2r)) \lesssim \mu(B(x,r))$ for all $x\in X$ and $r>0$.
However, only notational changes are required in the quasimetric case.
The quasimetric case of Theorem~\ref{thm:L1-FS} might be interesting in connection with the spline-like systems of H\"older continuous test functions that have been constructed in \cite{MR3008566}.

Since the intrinsic square function involves nontangentional regions, it is not difficult to deduce a sparse estimate from the weak type $(1,1)$ estimate.
Applying known estimates for sparse operators (see \cite{arXiv:1509.00273} for all cases with the exception of the two-weight weak type $(2,2)$ estimate which is proved in Appendix~\ref{sec:2weight-log} following the one-weight argument in \cite{MR3455749}) we obtain the following consequences.
\begin{corollary}
\label{cor:sq:weight}
Let $\omega$ be a modulus of continuity and $\beta\geq 1$.
Then for every $1<p<\infty$ and weights $w,\sigma$ we have
\begin{equation}
\label{eq:cor:strong}
\frac{\norm{G_{\omega,\beta}(\cdot\sigma)}_{L^{p}(\sigma) \to L^{p}(w)}}{\gamma^{1/2} \norm{\omega}_{\logDini}}
\lesssim
[w,\sigma]_{A_{p}}^{\frac1p} \cdot
\begin{cases}
[\sigma]_{A_{\infty}}^{\frac1p}, & 1<p\leq 2\\
[w]_{A_{\infty}}^{\frac12-\frac1p} + [\sigma]_{A_{\infty}}^{\frac1p}, & 2<p<\infty
\end{cases}
\end{equation}
and
\begin{equation}
\label{eq:cor:weak}
\frac{\norm{G_{\omega,\beta}(\cdot\sigma)}_{L^{p}(\sigma) \to L^{p,\infty}(w)}}{\gamma^{1/2} \norm{\omega}_{\logDini}}
\lesssim
[w,\sigma]_{A_{p}}^{\frac1p} \cdot
\begin{cases}
1, & 1<p< 2\\
(1+\log [w]_{A_{\infty}})^{1/2}, & p=2\\
[w]_{A_{\infty}}^{\frac12-\frac1p}, & 2<p<\infty.
\end{cases}
\end{equation}
\end{corollary}
The log-Dini norm can be replaced by the Dini norm under the same condition as in Theorem~\ref{thm:L1-FS}.
The $A_{p}$ and $A_{\infty}$ characteristics in Corollary~\ref{cor:sq:weight} have to be computed on adjacent dyadic grids.
The definition of the two weight $A_{p}$ characteristic is recalled in \eqref{eq:2weight-Ap}, while for the Fujii--Wilson $A_{\infty}$ characteristic we refer to \cite{arXiv:1509.00273}.
In the one weight case $w^{1/p}\sigma^{1/p'}\equiv 1$, estimating $[w]_{A_{\infty}} \lesssim [w]_{A_{p}}$ and $[\sigma]_{A_{\infty}} \lesssim [\sigma]_{A_{p'}} = [w]_{A_{p}}^{p'/p}$, the estimate \eqref{eq:cor:strong} recovers the bound
\begin{equation}
\label{eq:cor:Ap}
\frac{\norm{G_{\omega,\beta}}_{L^{p}(w) \to L^{p}(w)}}{\gamma^{1/2} \norm{\omega}_{\logDini}}
\lesssim
[w]_{A_{p}}^{\max(1/2,1/(p-1))}.
\end{equation}
The dependence on $\beta$ in \eqref{eq:cor:strong} and \eqref{eq:cor:weak} is optimal for power weights on $\R^{D}$; we include a short proof of this fact in Section~\ref{sec:beta-power}.

\section{Dyadic cubes}
A \emph{system of dyadic cubes} $\calD$ on $(X,\rho)$ with constants $\kappa>1$, $a_{0}>0$, $C_{1}<\infty$ consists of collections $\calD_{k}$, $k\in\Z$, of open subsets of $X$ such that
\begin{enumerate}
\item $\forall k\in\Z \quad \mu(X\setminus\cup_{Q\in\calD_{k}} Q)=0$,
\item If $l\geq k$, $Q\in\calD_{l}$, $Q'\in\calD_{k}$, then either $Q'\subseteq Q$ or $Q'\cap Q=\emptyset$,
\item For every $l\geq k$ and $Q'\in\calD_{k}$ there exists a unique $Q\in\calD_{l}$ such that $Q\supseteq Q'$,\
\item $\forall k\in\Z, Q\in\calD_{k} \quad \exists c_{Q}\in X : B(c_{Q},a_{0}\kappa^{k}) \subseteq Q \subseteq B(c_{Q}, C_{1} \kappa^{k})$.
\end{enumerate}
Abusing the notation we write $\scale(Q)=k$ if $Q\in\calD_{k}$, so that each cube remembers its scale although the same cube (as a set) may appear in other $\calD_{l}$.
We use $\calD$ to denote the disjoint union of $\calD_{k}$.
We also write $CQ=B(c_{Q},CC_{1}\kappa^{\scale(Q)})$, where $C\geq 1$.

A system of dyadic cubes in a general space of homogeneous type has been constructed by Christ \cite{MR1096400}, and we fix such a system $\calD$.
Since we do not require the small boundary property from the dyadic cubes, an easier construction by Hyt\"onen and Kairema \cite{MR2901199} also suffices for our purposes.
A finite collection of systems of dyadic cubes is called \emph{adjacent} if there exists $C<\infty$ such that for every ball $B(x,r)$ there exists a dyadic cube $Q$ in one of these collections such that $B(x,r) \subset Q \subset B(x,Cr)$.
Adjacent systems of dyadic cubes on geometrically doubling spaces, and in particular on spaces of homogeneous type, have been constructed in \cite{MR2901199}.

A collection $\sparse\subset\calD$ is called
\begin{enumerate}
\item \emph{$\eta$-sparse} if there exist pairwise disjoint subsets $E(Q)\subset Q\in\sparse$ with $\meas{E(Q)} \geq \eta \meas{Q}$ and
\item \emph{$\Lambda$-Carleson} if one has $\sum_{Q'\subset Q, Q'\in\sparse} \mu(Q') \leq \Lambda \mu(Q)$ for all $Q\in\calD$.
\end{enumerate}
It is known that a collection is $\eta$-sparse if and only if it is $1/\eta$-Carleson \cite[\textsection 6.1]{arXiv:1508.05639}.

\section{$L^{2}$ estimates}
\label{sec:square:L2}
\begin{lemma}
\label{lem:L2}
The operator
\[
A_{\omega} : L^{2}(X) \to L^{2}(X \times \Z)
\]
is bounded with norm $\lesssim(\Phi \norm{\omega}_{\Dini})^{1/2}$.
\end{lemma}
\begin{proof}
Let $(x,k)\to \phi^{(x,k)} \in C_{\omega}(x,k)$ be a linearizing function for the supremum in the definition of $A_{\omega}$.
It suffices to estimate the adjoint operator of the linearized version of $A_{\omega}$ that is given by
\[
\tilde A_{\omega}^{*} g(x) = \int_{X}\sum_{k\in\Z} g(y,k) \phi^{(y,k)}(x) \dif y.
\]
For each $k\in\Z$ and each dyadic cube $Q\in\calD_{k}$ write
\[
\int_{Q} g(y,\scale(Q)) \phi^{(y,k)}(x) \dif y = \lambda_{Q} a_{Q}(x),
\text{ where }
\lambda_{Q} = \Big( \int_{Q} \abs{g(y,\scale(Q))}^{2} \dif y \Big)^{1/2}
\]
and $a_{Q}\equiv 0$ if $\lambda_{Q}=0$.

Then the function $a_{Q}$ is supported inside $C Q$, has mean zero, and for all $x,x'\in CQ$ we have
\begin{align*}
\lambda_{Q} \abs{a_{Q}(x)-a_{Q}(x')}
&\leq
\int_{Q} \abs{g(y,\scale(Q))} \abs{\phi^{(y,\scale(Q))}(x)-\phi^{(y,\scale(Q))}(x')} \dif y\\
&\lesssim
\int_{Q} \abs{g(y,k)} \meas{Q}^{-1} \omega(\rho(x,x')/\kappa^{\scale(Q)}) \dif y\\
&\lesssim
\omega(\rho(x,x')/\kappa^{\scale(Q)}) \meas{Q}^{-1/2} \lambda_{Q}
\end{align*}
by the Cauchy--Schwarz inequality, so that
\[
\abs{a_{Q}(x)-a_{Q}(x')}
\lesssim
\meas{Q}^{-1/2} \omega(\rho(x,x')/\kappa^{\scale(Q)}).
\]
A similar argument shows $\sup \abs{a_{Q}} \lesssim \meas{Q}^{-1/2} \Phi$.
Following the argument in the proof of \cite[Theorem 1.14]{MR1107300} we will show that the collection of functions $(a_{Q})_{Q}$ is almost orthogonal.
Let $P,Q\in\calD$ be dyadic cubes with $\scale(P)\leq\scale(Q)$.
Then
\begin{align*}
\abs[\Big]{ \int_{X} a_{P}(x) \overline{a_{Q}(x)} \dif x }
&\leq
\int_{C P} \abs{a_{P}(x)} \abs{a_{Q}(x) - a_{Q}(x_{P})} \dif x\\
&\lesssim
\Phi \meas{P}^{-1/2} \int_{C P} \omega(\rho(x,x')/\kappa^{\scale(Q)}) \meas{Q}^{-1/2} \dif x\\
&\lesssim
\Phi \frac{\meas{P}^{1/2}}{\meas{Q}^{1/2}} \omega(\kappa^{\scale(P)-\scale(Q)}).
\end{align*}
Hence
\begin{align*}
\norm{\tilde A_{\omega}^{*} g}_{2}^{2}
=
\norm[\big]{ \sum_{Q\in\calD} \lambda_{Q} a_{Q} }_{2}^{2}
&\lesssim
\sum_{P,Q \in\calD : \scale(P)\leq\scale(Q)} \lambda_{P} \lambda_{Q} \abs[\big]{ \int a_{P} \overline{a_{Q}} }\\
&\lesssim
\Phi \sum_{P,Q : \scale(P)\leq\scale(Q), CP \cap CQ \neq \emptyset} \lambda_{P} \lambda_{Q} \frac{\meas{P}^{1/2}}{\meas{Q}^{1/2}} \omega( \kappa^{\scale(P)-\scale(Q)} ),
\end{align*}
and by Hölder's inequality this is bounded by
\begin{align*}
&\leq \Phi\cdot
\Big( \sum_{P} \lambda_{P}^{2} \sum_{Q : \scale(P)\leq\scale(Q), CQ \cap CP \neq \emptyset} \omega(\kappa^{\scale(P)-\scale(Q)}) \Big)^{1/2}
\Big( \sum_{Q} \lambda_{Q}^{2} \sum_{P : \scale(P)\leq\scale(Q), CQ \cap CP \neq \emptyset} \frac{\meas{P}}{\meas{Q}} \omega(\kappa^{\scale(P)-\scale(Q)}) \Big)^{1/2}\\
&\lesssim \Phi
\Big( \sum_{P} \lambda_{P}^{2} \sum_{k=0}^{\infty} \omega(\kappa^{-k}) \Big)^{1/2}
\Big( \sum_{Q} \lambda_{Q}^{2} \sum_{k=0}^{\infty} \omega(\kappa^{-k}) \Big)^{1/2}\\
&\leq \Phi \norm{\omega}_{\Dini}
\sum_{P} \lambda_{P}^{2}.
\end{align*}
Inserting the definition of $\lambda_{Q}$ we obtain the claim.
\end{proof}

\begin{proof}[Proof of \eqref{eq:L2-FS}]
By Fubini's theorem we can write
\begin{align*}
LHS\eqref{eq:L2-FS}
&\leq
\int_{X} v(x) \sum_{k\in\Z} \fint_{B(x,\beta \kappa^{k})} (A_{\omega}f(y,k))^{2}\\
&\lesssim
\int_{X} \sum_{k\in\Z} (A_{\omega}f(y,\kappa^{k}))^{2} \frac{v(B(y,\beta \kappa^{k}))}{\meas{B(y,\beta \kappa^{k})}} \dif y.
\end{align*}
Partitioning the region on which the integrand does not vanish into the sets
\[
F_{l} = \Set{(y,k)\in X\times\Z : 2^{l} < \frac{v(B(y,\beta \kappa^{k}))}{\meas{B(y,\beta \kappa^{k})}} \leq 2^{l+1}},
\quad k\in\Z,
\]
we obtain
\[
LHS\eqref{eq:L2-FS}
\lesssim
\sum_{l\in\Z} 2^{l} \int_{(y,k)\in F_{l}} (A_{\omega}f(y,k))^{2}.
\]
For every $(y,k)\in F_{l}$ the ball $B(y,\kappa^{k})$ is contained in the superlevel set of the maximal function $E_{l} := \Set{ Mv > 2^{l} }$ (this follows from the assumption $\beta\geq 1$).
The bounded support property of the test functions in the definition of $A_{\omega}$ implies that $A_{\omega}f(y,\kappa^{k}) = A_{\omega}(f\one_{E_{l}})(y,\kappa^{k})$.
Hence
\[
LHS\eqref{eq:L2-FS}
\lesssim
\sum_{l\in\Z} 2^{l} \int_{(y,k)\in F_{l}} (A_{\omega}(f\one_{E_{l}})(y,k))^{2}.
\]
Estimating the integral over $F_{l}$ by the integral over $X\times \Z$ and applying Lemma~\ref{lem:L2} we obtain
\begin{align*}
LHS\eqref{eq:L2-FS}
&\lesssim
\Phi \norm{\omega}_{\Dini} \sum_{l\in\Z} 2^{l} \norm{ f\one_{E_{l}} }_{2}^{2}\\
&=
\Phi \norm{\omega}_{\Dini} \int_{X} \abs{f}^{2} \sum_{l\in\Z} 2^{l} \one_{E_{l}}\\
&\lesssim
\Phi \norm{\omega}_{\Dini} \int_{X} \abs{f}^{2} Mv.
\qedhere  
\end{align*}
\end{proof}

\section{Weak type $(1,1)$ estimate}
\begin{proof}[Proof of \eqref{eq:isq-FS}]
Multiplying $\omega$ by a constant we may assume $\Phi \lesssim \omega(1) \lesssim \norm{\omega}_{\logDini}=1$.
Fix $\lambda>0$ and let $\calQ\subset\calD$ be a Whitney decomposition of the open set $\Omega := \Set{ Mf > \gamma^{-1/2}\lambda}$ with distance parameter $\beta$, that is, the collection of maximal dyadic cubes $Q\subset\Omega$ such that
\[
C\beta\diam(Q) \leq \dist(Q,X\setminus\Omega).
\]
The collection $\calQ$ covers $\Omega$ up to a set of measure $0$.
The corresponding Calder\'on--Zygmund decomposition of the function $f$ is given by
\[
f=g+b,
\quad
b=\sum_{Q\in\calQ} b_{Q},
\quad
b_{Q} = \one_{Q}\big(f-\mean{f}\big).
\]
For the good part $g$ we have then $\norm{g}_{\infty} \lesssim \gamma^{1/2} \lambda$ and $g \one_{Q} =\meas{Q}^{-1}\int_{Q}f$, and for the bad parts $b_{Q}$ we have $\supp b_{Q} \subset Q$, $\int b_{Q} = 0$, and $\int \abs{b_{Q}} \lesssim \gamma^{1/2} \lambda \meas{Q}$.
For $Q\in\calQ$ the ball $C\beta Q$ is still contained in $\Omega$.

By the Fefferman--Stein inequality \cite{MR0284802} we have
\begin{equation}
\label{eq:FS:exceptional}
v(\Omega) \lesssim \frac{\gamma^{1/2}}{\lambda} \int_{X} \abs{f} Mv.
\end{equation}
We estimate the terms
\begin{equation}
\label{eq:FS:good}
v(\Set{ G_{\omega,\beta}g > \lambda/2} \setminus\Omega)
\end{equation}
and
\begin{equation}
\label{eq:FS:bad}
v(\Set{ G_{\omega,\beta}b > \lambda/2} \setminus\Omega)
\end{equation}
separately.
Consider first the contribution of the good part.
By \eqref{eq:L2-FS} we obtain
\begin{align*}
\eqref{eq:FS:good}
&\lesssim
\lambda^{-2} \int_{X\setminus\Omega} (G_{\omega,\beta}g)^{2} v\\
&\lesssim
\lambda^{-2} \int_{X} \abs{g}^{2} M(\one_{X\setminus\Omega} v)\\
&\leq
\lambda^{-2} \Big(\gamma^{-1/2}\lambda \int_{X\setminus\Omega} \abs{f} M(v)
+
\sum_{Q} \mean{\abs{f}}^{2} \int_{Q\in\calQ} M(\one_{X\setminus C\beta Q}v) \Big)\\
&\lesssim
\frac{\gamma^{1/2}}{\lambda}
\Big( \int_{X\setminus\Omega} \abs{f} M(v)
+
\sum_{Q\in\calQ} \int_{Q} \abs{f} \inf_{Q} M(\one_{X\setminus C\beta Q}v) \Big)\\
&\leq
\frac{\gamma^{1/2}}{\lambda} \int_{X} \abs{f} M(v).
\end{align*}
It remains to estimate the contribution of the bad part.
Fix $Q\in\calQ$, then for $(x,k)\in X\times\Z$ with $k\geq\scale(Q)$ we have
\begin{align*}
A_{\omega}b_{Q}(x,k)
&\lesssim
\abs[\Big]{ \int_{Q} \phi^{(x,k)}(y) b_{Q}(y) \dif y }\\
&\leq
\int_{Q} \abs{\phi^{(x,k)}(y) - \phi^{(x,k)}(c_{Q})} \abs{b_{Q}(y)} \dif y\\
&\leq
\meas{B(x,\kappa^{k})}^{-1} \int_{Q} \omega(\abs{y-c_{Q}}/t) \abs{b_{Q}(y)} \dif y\\
&\lesssim
\meas{B(c_{Q},\kappa^{k})}^{-1} \omega(\kappa^{\scale(Q)-k}) \norm{b_{Q}}_{1}
\end{align*}
and $A_{\omega}b_{Q}(x,k) = 0$ if $\dist(x,Q)>\kappa^{k}$.

It follows that for every $x\not\in C\beta Q$ we have
\begin{align*}
(G_{\omega,\beta}b_{Q}(x))^{2}
&=
\sum_{k\in\Z : Q \cap B(x,(\beta+1)\kappa^{k}) \neq\emptyset} \fint_{B(x,\beta \kappa^{k})} (A_{\omega}b_{Q}(y,k))^{2} \dif y\\
&\lesssim
\sum_{k : \dist(x,Q) \leq (\beta+1)\kappa^{k}}
\meas{B(c_{Q},\beta\kappa^{k})}^{-1} \int_{X} (A_{\omega}b_{Q}(y,k))^{2} \dif y\\
&\lesssim
\sum_{k : \dist(x,Q) \leq (\beta+1)\kappa^{k}} \frac{\meas{B(c_{Q},\kappa^{k})}}{\meas{B(c_{Q},\beta\kappa^{k})}}
(\meas{B(c_{Q},\kappa^{k})}^{-1} \omega(\kappa^{\scale(Q)-k}) \norm{b_{Q}}_{1})^{2}\\
&\lesssim
\gamma \norm{b_{Q}}_{1}^{2}
\sum_{k : \dist(x,Q) \leq (\beta+1)\kappa^{k}} \frac{\omega(\kappa^{\scale(Q)-k})^{2}}{\meas{B(c_{Q},\beta\kappa^{k})}^{2}}.
\end{align*}
Taking the square root in the above inequality and integrating it in $x$ we obtain
\begin{align*}
\int_{X\setminus C\beta Q} G_{\omega,\beta}(b_{Q})(x) v(x) \dif x
&\lesssim
\gamma^{1/2} \norm{b_{Q}}_{1}
\int_{X\setminus \tilde Q} v(x) \Big( \sum_{k : \dist(x,Q) \leq (\beta+1)\kappa^{k}} \frac{\omega(\kappa^{\scale(Q)-k})^{2}}{\meas{B(c_{Q},\beta\kappa^{k})}^{2}} \Big)^{1/2} \dif x\\
&\leq
\gamma^{1/2} \norm{b_{Q}}_{1}
\int_{X\setminus C\beta Q} v(x) \sum_{k : \dist(x,Q) \leq (\beta+1)\kappa^{k}} \frac{\omega(\kappa^{\scale(Q)-k})}{\meas{B(c_{Q},\beta\kappa^{k})}} \dif x\\
&\sim
\gamma^{1/2} \norm{b_{Q}}_{1}
\sum_{l\geq \scale(Q)} \int_{\dist(x,Q) \sim \beta \kappa^{l}} v(x) \dif x \sum_{k \geq l} \frac{\omega(\kappa^{\scale(Q)-k})}{\meas{B(c_{Q},\beta\kappa^{k})}}\\
&\lesssim
\gamma^{1/2} \norm{b_{Q}}_{1} (\inf_{Q} Mv)
\sum_{l\geq \scale(Q)} \meas{B(c_{Q},\beta\kappa^{l})} \sum_{k \geq l} \frac{\omega(\kappa^{\scale(Q)-k})}{\meas{B(c_{Q},\beta\kappa^{k})}}\\
&\lesssim
\gamma^{1/2} \int_{Q} \abs{f} Mv
\sum_{k:k\geq \scale(Q)} \fbox{$\sum_{l:k\geq l\geq \scale(Q)} \frac{\meas{B(c_{Q},\beta\kappa^{l})}}{\meas{B(c_{Q},\beta\kappa^{k})}}$} \omega(\kappa^{\scale(Q)-k})\\
&\leq
\gamma^{1/2} \int_{Q} \abs{f} Mv
\sum_{k:k\geq \scale(Q)} (k-\scale(Q)+1) \omega(\kappa^{\scale(Q)-k})\\
&\sim
\gamma^{1/2} \int_{Q} \abs{f} Mv \norm{\omega}_{\logDini}.
\end{align*}
With the reverse doubling condition the sum in the box is bounded uniformly in $k,l,Q$, and we obtain an estimate in terms of the Dini instead of the log-Dini norm of $\omega$.
Summing this inequality over $Q\in\calQ$ and using subadditivity of the intrinsic square function we obtain
\[
\eqref{eq:FS:bad}
\lesssim
\lambda^{-1} \int_{X\setminus \Omega} G_{\omega,\beta}(b)(x) v(x) \dif x
\lesssim
\frac{\gamma^{1/2}}{\lambda} \int_{\Omega} \abs{f} Mv.
\]
Summing the contributions of \eqref{eq:FS:exceptional}, \eqref{eq:FS:good}, and \eqref{eq:FS:bad} we obtain the claim.
\end{proof}

\section{Sparse domination}
\label{sec:square:sparse}
\begin{proof}[Proof of Corollary~\ref{cor:sq:weight}]
Without loss of generality assume $\beta = \kappa^{\Delta k}$ with $\Delta k\in\N$.
Define a function on pairs of nested cubes by
\[
F(Q',Q) := \sup_{x'\in Q'} \sum_{k=\scale(Q')}^{\scale(Q)} \fint_{B(x',C\kappa^{k})} (A_{\omega}f(y,k-\Delta k))^{2} \dif y,
\quad
Q'\subset Q,\,Q,Q'\in\calD.
\]
Then
\[
G_{\omega,\beta}f(x)^{2}
\lesssim
\sup_{x\in Q'\subset Q} F(Q',Q)
\lesssim
G_{\omega,C\beta}f(x)^{2}.
\]
Fix $0<\eta<1$.
It suffices to show that for every $k_{0}\in\N$ there exists an $\eta$-sparse collection $\sparse = \sparse^{k_{0}}\subset\calD$ such that
\begin{equation}
\label{eq:sparse-dom}
\sup_{x\in Q'\subset Q, -k_{0}\leq \scale(Q')\leq \scale(Q)\leq k_{0}} F(Q',Q)
\lesssim
\frac{\gamma \norm{\omega}_{\logDini}^{2}}{(1-\eta)^{2}} \sum_{Q\in\sparse} \one_{Q}(x) \Big(\inf_{c\in\C} \fint_{CQ} \abs{f-c}\Big)^{2}
\end{equation}
with an implicit constant that does not depend on $k_{0}$.
Indeed, the sparse square function on the right-hand side of \eqref{eq:sparse-dom} can be estimated by a finite sum of sparse square functions associated to adjacent dyadic grids in which the oscillation is taken over $Q$ instead of a dilated cube, see \cite{MR3484688}.
The estimates from \cite{arXiv:1509.00273} and Theorem~\ref{thm:2weight-log} then apply.

We construct the collection $\sparse$ inductively starting with $\calP_{0} = \calD_{k_{0}}$.
Suppose that $\calP_{n}$ has been constructed for some $n$.
For each $P\in\calP_{n}$ let $\ch(P)$ denote the collection of maximal cubes $P'\subset P$ with $\scale(P')\geq -k_{0}$ such that
\begin{equation}
\label{eq:sparse-dom:stop}
F(P',P)^{1/2}
\geq
\frac{C \gamma^{1/2} \norm{\omega}_{\logDini}}{1-\eta} \inf_{c\in\C}\int_{CP} \abs{f-c}.
\end{equation}
It follows from Theorem~\ref{thm:L1-FS} that
\begin{equation}
\label{eq:sparse-dom:sparse}
\sum_{P'\in\ch(P)} \mu(P')
\leq
(1-\eta) \mu(P)
\end{equation}
provided that $C$ in \eqref{eq:sparse-dom:stop} is large enough (notice that $F(P',P)$ only depends on the values of $f$ on $CP$ and does not change upon adding a constant to $f$).
Let $\calP_{n+1} = \cup_{P\in\calP_{n}} \ch(P)$.
This procedure stops after finitely many steps and the required sparse collection is given by $\sparse = \cup_{n} \calP_{n}$.
In order to verify \eqref{eq:sparse-dom} let $x\in X$ and $P_{0}\supset \dotsb \supset P_{N} \ni x$ be the maximal chain of stopping cubes that contain $x$ and $P'$ the cube of scale $-k_{0}$ that contains $x$.
Then
\[
\sup_{x\in Q'\subset Q, -k_{0}\leq \scale(Q')\leq \scale(Q)\leq k_{0}} F(Q',Q)
=
P(P',P_{N})
+
\sum_{i=0}^{N-1} F(\widehat{P_{i+1}},P_{i}),
\]
where $\widehat{P}$ denotes the parent of $P$.
Notice that $P'$ is not a stopping child of $P_{N}$, since otherwise by \eqref{eq:sparse-dom:sparse} we have $P'\subsetneq P_{N}$, contradicting minimality of $P_{N}$ in $\sparse$.
The estimate \eqref{eq:sparse-dom} now follows from the fact that the stopping condition \eqref{eq:sparse-dom:stop} does not hold for pairs $(P',P_{N})$ and $(\widehat{P_{i+1}},P_{i})$.
\end{proof}

\begin{remark}
The fact that the oscillation can be used in place of the average of $f$ over $CQ$ in \eqref{eq:sparse-dom} has been observed in a different situation in \cite{arxiv:1703.00228}.
While this preserves some cancellation of the square function in the sparse operator, weighted estimates do not seem to benefit from this refinement.
\end{remark}

\appendix
\section{A logarithmically bumped two weight estimate}
\label{sec:2weight-log}
In this appendix we repeat an argument from \cite{MR3455749} in the two weight setting to complement \cite[Theorem 1.2]{arXiv:1509.00273} with a weak type $(r,r)$ estimate.
The result below does not depend neither on a metric structure nor on a doubling hypothesis.
In this section $(X,\mu)$ denotes a measure space and $\calD$ a dyadic grid, that is, a collection of measurable subsets of $X$ such that for all $Q,Q'\in\calD$ we have $0<\meas{Q}<\infty$ and $Q\cap Q' \in \Set{Q,Q',\emptyset}$.
Recall that the dyadic $A_{p}$ characteristic of a pair of weights is defined by
\begin{equation}
\label{eq:2weight-Ap}
[w,\sigma]_{A_{p}} := \sup_{Q\in\calD} \mean{w} \mean{\sigma}^{p-1}.
\end{equation}
\begin{theorem}
\label{thm:2weight-log}
Let $\sparse\subset\calD$ be an $\eta$-sparse collection for some $0<\eta<1$, and let $1<p<\infty$.
Let $w,\sigma$ be weights on $X$ and assume that the reverse H\"older inequality
\begin{equation}
\label{eq:rev-holder}
(w^{r})_{Q}^{1/r}
\leq
C (w)_{Q}
\end{equation}
holds for some $r>1$, $C<\infty$, and all $Q\in\sparse$.
Then
\[
\norm{ A_{\sparse}^{p} (\cdot\sigma) }_{L^{p}(\sigma) \to L^{p,\infty}(w)}
\lesssim_{p,\eta}
(1+\log r')^{1/p} [w,\sigma]_{A_{p}}^{1/p},
\]
where
\[
A_{\sparse}^{p} (f)
=
\Big( \sum_{Q\in\sparse} (f)_{Q}^{p} \one_{Q} \Big)^{1/p}.
\]
\end{theorem}

\begin{remark}
The sharp reverse Hölder inequality on the dyadic grid in $\R^{d}$, proved in \cite[Theorem 2.3]{MR2990061}, tells that \eqref{eq:rev-holder} holds in this case with $C=2$ and $r'=2^{d+1} [w]_{A_{\infty}}$.
\end{remark}
As already observed in \cite{MR3455749}, the argument below, with the Fefferman--Stein inequality \cite{MR0284802} in place of Muckenhoupt's two weight weak type $(p,p)$ inequality, also gives a short proof of the bound
\[
\norm{ A_{\sparse}^{1} f }_{L^{1,\infty}(w)} \lesssim_{\eta} (1+\log r') \int \abs{f} Mw.
\]
Stronger estimates for $A_{\sparse}^{1}$ have been obtained in \cite{MR3327006} and \cite{MR3455749}.

\begin{proof}[Proof of Theorem~\ref{thm:2weight-log}]
Partitioning $\sparse$ into $O(\frac1\eta)$ families we may assume that $\sparse$ is $\frac78$-sparse, see \cite[\textsection 6]{arXiv:1508.05639} (in fact the rather sophisticated argument given there is not strictly necessary because we can assume without loss of generality that $\sparse$ is finite).
By homogeneity it suffices to show
\begin{equation}
\label{eq:2weight-log:claim}
w\{ (A_{\sparse}^{p}(f\sigma))^{p} > 2 \}
\lesssim
(1+\log r') [w,\sigma]_{A_{p}} \norm{f}_{L^{p}(\sigma)}^{p}.
\end{equation}
Let
\[
\sparse_{m} := \{ Q\in \sparse : 2^{-m-1} < \mean{f\sigma} \leq 2^{-m}\},
\quad
m\in\Z,
\]
then the left-hand side of \eqref{eq:2weight-log:claim} is bounded by
\begin{equation}
\label{eq:2weight-log:claim2}
w\{ M(f\sigma)>1\}
+
w\{ \sum_{m=0}^{m_{0}} (A_{\sparse_{m}}^{p}(f\sigma))^{p} > 1 \}
+
w\{ \sum_{m=m_{0}}^{\infty} (A_{\sparse_{m}}^{p}(f\sigma))^{p} > 1 \},
\end{equation}
where $Mf(x)=\sup_{x\in Q\in\calD} (f)_{Q}$ denotes the dyadic maximal function and $m_{0}\in\N$ will be chosen later.

We recall from \cite[\textsection 7]{MR0293384} the two weight weak type estimate for $M$ and its proof: let $\lambda>0$ and $\calQ\subset\calD$ a collection of disjoint dyadic cubes such that $(f\sigma)_{Q} > \lambda$ for $Q\in\calQ$.
Then
\begin{align*}
\lambda^{p} \sum_{Q\in\calQ} w(Q)
&\leq
\sum_{Q\in\calQ} w(Q) (f\sigma)_{Q}^{p}\\
&\leq
\sum_{Q\in\calQ} \mu(Q) (w)_{Q} (f^{p}\sigma)_{Q} (\sigma)_{Q}^{p/p'}\\
&\leq
[w,\sigma]_{A_{p}} \sum_{Q\in\calQ} \int_{Q} f^{p} \sigma\\
&\leq
[w,\sigma]_{A_{p}} \norm{f}_{L^{p}(\sigma)}^{p},
\end{align*}
so that
\begin{equation}
\label{eq:Muck-weak-max}
w\Set{Mf>\lambda} \leq \lambda^{-p} [w,\sigma]_{A_{p}} \norm{f}_{L^{p}(\sigma)}^{p}.
\end{equation}
This gives us the required estimate for the first term in \eqref{eq:2weight-log:claim2}.

Next we claim that for every $m$
\begin{equation}
\label{eq:single-slice}
\norm{A_{\sparse_{m}}^{p}(f\sigma)}_{L^{p}(w)} \lesssim [w,\sigma]_{A_{p}}^{1/p} \norm{f}_{L^{p}(\sigma)}
\end{equation}
holds.
To see this, let
\[
E_{m}(Q) := Q \setminus \cup_{Q'\subsetneq Q, Q'\in \sparse_{m}} Q',
\quad
Q\in \sparse_{m}.
\]
Then it follows from $\frac78$-sparsity that
\[
\mean{f\sigma} \sim \mean{f\sigma 1_{E_{m}(Q)}}.
\]
Using this fact, Hölder's inequality, and disjointness of the sets $E_{m}(Q)$ we obtain
\begin{align*}
\sum_{Q\in \sparse_{m}} \mean{f\sigma}^{p} w(Q)
&\sim
\sum_{Q\in \sparse_{m}} \mean{f\sigma 1_{E_{m}(Q)}}^{p} w(Q)\\
&\leq
\sum_{Q\in \sparse_{m}} \mean{f^{p}\sigma 1_{E_{m}(Q)}} \mean{\sigma}^{p/p'} w(Q)\\
&\leq
[w,\sigma]_{A_{p}} \sum_{Q\in \sparse_{m}} \int f^{p}\sigma 1_{E_{m}(Q)}\\
&\leq
[w,\sigma]_{A_{p}} \int f^{p}\sigma,
\end{align*}
finishing the proof of \eqref{eq:single-slice}.
The second term in \eqref{eq:2weight-log:claim2} can now be estimated by
\[
\norm{ \sum_{m=0}^{m_{0}} (A_{\sparse_{m}}^{p}(f\sigma))^{p} }_{L^{1}(w)}
\leq
\sum_{m=0}^{m_{0}} \norm{ (A_{\sparse_{m}}^{p}(f\sigma))^{p} }_{L^{1}(w)}
=
\sum_{m=0}^{m_{0}} \norm{ A_{\sparse_{m}}^{p}(f\sigma) }_{L^{p}(w)}^{p},
\]
and in view of \eqref{eq:single-slice} this is acceptable provided $m_{0} \lesssim \log r'$.

It remains to control the third term in \eqref{eq:2weight-log:claim2} under the hypothesis $m_{0} \gtrsim_{p} \log r'$.
To this end let $\sparse_{m}^{*}$ consist of the maximal cubes in $\sparse_{m}$.
By definition of $\sparse_{m}$ the third term in \eqref{eq:2weight-log:claim2} is controlled by
\begin{align*}
w\{\sum_{m=m_{0}}^{\infty} 2^{-pm} \sum_{Q\in \sparse_{m}} 1_{Q} >  c \sum_{m=m_{0}}^{\infty} m^{-2}\}
&\leq
\sum_{m=m_{0}}^{\infty} w\{ \sum_{Q\in \sparse_{m}} 1_{Q} >  c 2^{pm} m^{-2}\}\\
&\leq
\sum_{m=m_{0}}^{\infty} \sum_{Q \in \sparse_{m}^{*}} w(\beta_{Q}),
\end{align*}
where $\beta_{Q} = Q \cap \{ \sum_{Q'\in \sparse_{m}} 1_{Q'} >  c 2^{pm} m^{-2}\}$.
By sparsity we can estimate the reference measure of $\beta_{Q}$ by
\[
\meas{\beta_{Q}} \lesssim \exp(-c 2^{cm}) \abs{Q}
\]
By Hölder's inequality and \eqref{eq:rev-holder} we obtain
\[
\mean{w1_{\beta(Q)}}
\leq
\mean{1_{\beta(Q)}}^{1/r'} \mean{w^{r}}^{1/r}
\lesssim
\exp(-c 2^{cm}/r') \mean{w}.
\]
Multiplying this by $\abs{Q}$ and summing over $Q\in \sparse_{m}^{*}$ and $m$ we obtain the estimate
\begin{align*}
\MoveEqLeft
\sum_{m=m_{0}}^{\infty}  \exp(-c 2^{cm}/r') \sum_{Q\in \sparse_{m}^{*}} w(Q)\\
&\leq
\sum_{m=m_{0}}^{\infty} \exp(-c 2^{cm}/r') w\{ M(f\sigma) > 2^{-m-1}\}\\
&\lesssim
[w,\sigma]_{A_{p}} \norm{f}_{L^{p}(\sigma)}^{p} \sum_{m=m_{0}}^{\infty} 2^{pm} \exp(-c 2^{cm}/r'),
\end{align*}
where we have used \eqref{eq:Muck-weak-max} in the last line.
The hypothesis on $m_{0}$ ensures that the latter sum is $\lesssim 1$.
\end{proof}

\section{Optimality of the dependence on the aperture}
\label{sec:beta-power}
Let $\omega(t)=t$ and $X=\R^{d}$.
In this case $\gamma = \beta^{d}$.
In this appendix we show that in this case the exponent of $\gamma$ in \eqref{eq:cor:Ap} cannot be improved using an observation going back at least to \cite[\textsection 2]{MR0257819}.

Let $f$ be a function on $\R^{d}$ and $\phi\in C_{\omega}(0,0)$ be such that $\abs{\int f \phi}=c\neq 0$.
Then
\[
A_{\omega}f(y,k) \geq (\kappa^{k})^{-d-1} c
\quad\text{provided } \kappa^{k}>\abs{y}+1,
\]
since then $B(0,1)\subset B(y,\kappa^{k})$, and the convolution with a suitable multiple of $\phi$ occurs in the supremum defining $A_{\omega}$.
In particular, for $\beta = \kappa^{k_{0}}$, $k_{0}\gg 0$, and $\abs{x} \sim \kappa^{k}$, $k\geq k_{0}$, we obtain
\begin{multline*}
(G_{\omega,C\beta}f(x))^{2}
\geq
\kappa^{-kd} \int_{B(0,\kappa^{k})} (A_{\omega}f(y,k-k_{0}))^{2} \dif y\\
\geq
\kappa^{-kd} \int_{B(0,\kappa^{k-k_{0}}-1)} (A_{\omega}f(y,k-k_{0}))^{2} \dif y
\gtrsim
\kappa^{-k_{0}d} \kappa^{(k-k_{0})(-d-1)}
\sim
\beta^{-d} (\abs{x}/\beta)^{-2(d+1)}.
\end{multline*}
Let $w(x)=\abs{x}^{\alpha}$, $-d<\alpha<(p-1)d$, so that $w$ is an $A_{p}$ weight.
Then we obtain
\[
\norm{ G_{\omega,\beta}f }_{L^{p,\infty}(w)}^{p}
\gtrsim
\beta^{(-d/2)p}
w\{ \abs{x}\sim\beta \}
\gtrsim
\beta^{d+\alpha-dp/2}.
\]
Taking $\alpha$ close to $(p-1)d$ we see that the power of $\gamma$ in \eqref{eq:cor:strong}, \eqref{eq:cor:weak}, and \eqref{eq:cor:Ap} cannot be decreased.
\printbibliography
\end{document}
